%% file: PBWbranesGiov.tex
\documentclass{amsart}

\usepackage{latexsym,amsmath,amstext,mathrsfs,amsthm,amscd,amsopn,verbatim,amsrefs}

\usepackage{graphicx}
\usepackage[dvips]{color}
\usepackage{epsfig}
\usepackage{psfrag}
\usepackage[all]{xy}
\usepackage{amscd}
\usepackage{amssymb}

\newtheorem{Thm}{Theorem}[section]
\newtheorem{Prop}[Thm]{Proposition}
\newtheorem{Lem}[Thm]{Lemma}

\theoremstyle{remark}
\newtheorem{Rem}[Thm]{Remark}
\newtheorem*{Ack}{Acknowledgment}

\theoremstyle{definition}

\newtheorem{Def}[Thm]{Definition}

\title[Deformation quantization with generators and relations]{Deformation quantization with generators and relations}
\author{Damien Calaque}
\author{Giovanni Felder}
\author{Carlo A. Rossi}
\address{D.C.: Institut Camille Jordan, Universit\'e Claude Bernard Lyon 1, 43
boulevard du 11 novembre 1918, F-69622 Villeurbanne Cedex France}
\address{G. F: Department of mathematics, ETH Zurich,
8092 Zurich, Switzerland}
\address{C.~A.~R.: Centro de An\'{a}lise Matem\'{a}tica, Geometria e Sistemas Din\^{a}micos,
Departamento de Matem\'atica, Instituto Superior T\'ecnico, Av.
Rovisco Pais, 1049-001 Lisboa, Portugal}

\thanks{This work has been partially supported by SNF
Grant 200020-122126}

\begin{document}


\begin{abstract}
In this paper we prove a conjecture of B.~Shoikhet which claims that two quantization procedures arising from Fourier dual constructions actually coincide.
\end{abstract}
\maketitle

\section{Introduction}

There are two ways to quantize a polynomial Poisson structure $\pi$ on the dual $V^*$ of a finite dimensional complex vector space $V$, using Kontsevich's formality as a starting point.

\medskip

The first (obvious) way is to consider the image $\mathcal U(\pi_\hbar)$ of $\pi_\hbar=\hbar\pi$ through Kontsevich's $L_\infty$-quasi-isomorphism
$$
\mathcal U:\mathrm T_{\rm poly} (V^*)\longrightarrow \mathrm D_{\rm poly}(V^*)\,,
$$
and to take ${\rm m}_\star:={\rm m}+\mathcal U(\pi_\hbar)$ as a $\star$-product quantizing $\pi$, $\mathrm m$ being the standard product on $\mathrm S(V)=\mathcal O_{V^*}$.

\medskip

The main idea, due to B.~Shoikhet \cite{Sh}, behind the second (less obvious) way is to deform the relations of $\mathrm S(V)$ instead of the
product ${\rm m}$ itself. 

Consider for example a constant Poisson structure $\pi$ on $V^*$: the deformation quantization of $\mathrm S(V)$ w.r.t.\ $\hbar\pi$ is the Moyal--Weyl algebra $\mathrm S(V)[\![\hbar]\!]$ with Moyal product $\star$ given by
\[
f_1\star f_2=\mathrm m\!\left(\exp{\frac{\hbar\pi}2}(f_1\otimes f_2)\right), 
\]
where $\pi$ is understood here as a bidifferential endomorphism of $\mathrm S(V)\otimes\mathrm S(V)$.
On the other hand, it is well-known that the Moyal--Weyl algebra associated to $\pi$ is isomorphic to the free associative algebra over $\mathbb C[\![\hbar]\!]$ with generators $x_i$ (for $\{x_i\}$ a basis of $V$) by the relations
\[
x_i\star x_j-x_j\star x_i=\hbar \pi_{ij}.
\]
The construction we are interested in generalizes to any polynomial Poisson structure $\pi$ on $V^*$ the two ways of characterizing the Moyal--Weyl algebra associated to $\pi$.

More conceptually, $\mathrm S(V)$ is a quadratic Koszul algebra of the form $T(V)/\langle R\rangle$, where $R$ is the subspace of $V^{\otimes 2}$ spanned by vectors of the form $x_i\otimes x_j-x_j\otimes x_i$, $\{x_i\}$ as in the previous paragraph.
The right-hand side of the identity $\mathrm S(V)=\mathrm T(V)/\langle R\rangle$ can be viewed as the $0$-th cohomology of the free associative dg (short for differential graded from now on) algebra $\mathrm T(\wedge^-(V))$ over $\mathbb C$, where $\wedge^-(V)$ is the graded vector space $\wedge^-(V)=\bigoplus_{p=-d+1}^0 \wedge^-(V)_p=\bigoplus_{p=-d+1}^0 \wedge^{-p+1}(V)$ and differential $\delta$ on generators $\{x_{i_1},x_{i_1,i_2},\dots\}$ of $\wedge^-(V)$ specified by 
\[
\delta (x_{i_1})=0,\ \delta (x_{i_1,i_2})=x_{i_1}\otimes x_{i_2}-x_{i_2}\otimes x_{i_1},\ \text{etc.}
\]
Observe that the differential $\delta$ dualizes the product of the graded commutative algebra $\wedge(V^*)$: in fact, $\wedge(V^*)$ is the Koszul dual of $\mathrm S(V)$, and the above complex comes from the identification $\mathrm S(V)=\mathrm{Ext}_{\wedge(V^*)}(\mathbb C,\mathbb C)$ by explicitly computing the cohomology on the right-hand side w.r.t.\ the bar resolution of $\mathbb C$ as a (left) $\wedge(V^*)$-module (the above dg, short for differential graded, algebra is the cobar construction of $\mathrm S(V)$, and $\delta$ is the cobar differential).
The above dg algebra is acyclic except in degree $0$; the $0$-th cohomology is readily computed from the above formul\ae\ and equals precisely $\mathrm T(V)/\langle R\rangle$.

Therefore, the idea is to prove that the property of being Koszul and Koszul duality between $\mathrm S(V)$ and $\wedge(V^*)$ is preserved (in a suitable sense, which will be clarified later on) by deformation quantization.

Namely, one makes use of the graded version~\cite{CF} of Kontsevich's formality theorem, applied to the Fourier dual space $V[1]$.
We then have an $L_\infty$-quasi-isomorphism
$$
\mathcal V:T_{\rm poly}(V^*)\cong T_{\rm poly}(V[1])\longrightarrow D_{\rm poly}(V[1])\,,
$$
and the image $\mathcal V(\widehat{\pi_\hbar})$ of
$\widehat{\pi_\hbar}$, where $\widehat{\bullet}$ is the isomorphism
$\mathrm T_{\rm poly}(V^*)\cong \mathrm T_{\rm poly}(V[1])$ of dg Lie
algebras (graded Fourier transform), defines a deformation quantization of the 
graded commutative algebra $\wedge(V^*)$ as a (possibly curved) $A_\infty$-algebra.

In the context of the Formality Theorem with $2$ branes~\cite{CFFR}, the deformation quantization of $\wedge(V^*)$ is Koszul dual (in a suitable sense) w.r.t.\ the first deformation quantization of $\mathrm S(V)$, and the (possibly curved) $A_\infty$-structure on the deformation quantization of $\wedge(V^*)$ induces a deformation $\delta_\hbar$ of the cobar
differential $\delta$, which in turn produces a deformation $\mathcal I_\star$ of the two-sided ideal
$\mathcal I=\langle R\rangle$ in $\mathrm T(V)$ of defining relations of $\mathrm
S(V)$.

We are then able to prove the following result, first conjectured by
Shoikhet in \cite[Conjecture 2.6]{Sh2}:
\begin{Thm}[see Theorem \ref{t-sh}]\label{main}
Given a polynomial Poisson structure $\pi$ on $V^*$ as above, the algebra $A_\hbar:=\big(\mathrm S(V)[\![\hbar]\!],{\rm m}_\star\big)$
is isomorphic to the quotient of $\mathrm T(V)[\![\hbar]\!]$ by the
two-sided ideal $\mathcal I_\star$; the isomorphism is an
$\hbar$-deformation of the standard symmetrization map from $\mathrm
S(V)$ to $\mathrm T(V)$.
\end{Thm}
\begin{Rem}\label{r-formal}
We mainly consider here a formal polynomial Poisson structure of the form $\hbar\pi$, but all the arguments presented here apply as well to any formal polynomial Poisson structure $\pi_\hbar=\hbar\pi_1+\hbar^2\pi_2+\cdots$, where $\pi_i$ is a polynomial bivector field.
\end{Rem}

The paper is organized as follows. In Section 2 we start with a
recollection on $A_\infty$-algebras and bimodules.  We then
formulate the formality theorem with two branes of \cite{CFFR} in a
form suitable for the application at hand.  After this we describe
the deformation of the cobar complex obtained from $\mathcal
V(\widehat{\pi_\hbar})$ and prove Theorem \ref{main}. We conclude
the paper with three examples, see Section~\ref{s-3}: the cases of constant,
linear, and quadratic Poisson structures.

\begin{Ack}
We express our gratitude to the anonymous referee for the careful reading of the manuscript and for many useful comments and suggestions, which have helped us improve the paper.
\end{Ack}

\section{A deformation of the cobar construction of the exterior coalgebra}\label{s-2}

\subsection{$A_\infty$-algebras and (bi)modules of finite type}\label{ss-2-1}
We first recall the basic notions of the theory of $A_\infty$-algebras
and modules, see \cite{Keller, CFFR} to fix the conventions and
settle some finiteness issues. Note that we allow non-flat
$A_\infty$-algebras in our definition. Let $\mathrm{T}(V)=\mathbb
C\oplus V\oplus V^{\otimes 2}\oplus\cdots$ be the tensor coalgebra
of a $\mathbb Z$-graded complex vector space $V$ with coproduct
$\Delta(v_1,\dots,v_n)=\sum_{i=0}^n(v_1,\dots,v_i)\otimes(v_{i+1},\dots,v_n)$
and counit $\eta(1)=1$, $\eta(v_1,\dots,v_n)=0$ for $n\geq 1$. Here
we write $(v_1,\dots,v_n)$ as a more transparent notation for
$v_1\otimes\cdots\otimes v_n\in \mathrm{T}(V)$ and set
$()=1\in\mathbb C$. Let $V[1]$ be the graded vector space with
$V[1]^i=V^{i+1}$ and let the suspension $s\colon V\to V[1]$ be the
map $a\mapsto a$ of degree $-1$. Then an $A_\infty$-algebra over
$\mathbb C$ is a $\mathbb Z$-graded vector space $B$ together with a codifferential
$\mathrm{d}_B\colon \mathrm{T}(B[1])\to \mathrm{T}(B[1])$, namely a
linear map of degree 1 which is a coderivation of the coalgebra and
such that $\mathrm{d}_B\circ \mathrm{d}_B=0$. A coderivation is
uniquely given by its components $\mathrm{d}_B^k\colon B[1]^{\otimes
k}\to B[1]$, $k\geq0$ and any set of maps $\colon B[1]^{\otimes
k}\to B[1]$ of degree $1$ uniquely extends to a coderivation. This
coderivation is a codifferential if and only if $\sum_{j+k+l=n}
\mathrm{d}_B^n\circ(\mathrm{id}^{\otimes j}\otimes
\mathrm{d}_B^k\otimes \mathrm{id}^{\otimes l})=0$ for all $n\geq0$.
The maps $\mathrm{d}_B^k$ are called {\em Taylor components} of the
codifferential $\mathrm{d}_B$. If $\mathrm{d}_B^0=0$, the
$A_\infty$-algebra is called {\em flat}. Instead of $\mathrm{d}_B^k$
it is convenient to describe $A_\infty$-algebras through the
product maps $\mathrm{m}^k_B=s^{-1}\circ \mathrm{d}_B^k\circ
s^{\otimes k}$ of degree $2-k$. If $\mathrm{m}_B^k=0$ for all $k\neq
1,2$ then $B$ with differential $\mathrm{m}^1_B$ and product
$\mathrm{m}^2_B$ is a differential graded algebra. A {\em unital}
$A_\infty$-algebra is an $A_\infty$-algebra $B$ with an element
$1\in B^0$ such that
\[
\begin{array}{ll}
\mathrm{m}_B^2(1,b)=\mathrm{m}_B^2(b,1)=b,& \forall b\in B,\\
\mathrm{m}_B^j(b_1,\dots,b_j)=0,& \text{if $b_i=1$ for some $1\leq
i\leq j$ and $j\neq 2$.}
\end{array}
\]
The first condition translates to
$\mathrm{d}_B^2(s1,b)=b=(-1)^{|b|-1}\mathrm{d}_B^2(b,s1)$, if $b\in
B[1]$ has degree $|b|$. A {\em right $A_\infty$-module $M$} over an
$A_\infty$-algebra $B$ is a graded vector space $M$ together with a
degree one codifferential $\mathrm{d}_M$ on the cofree right
$\mathrm{T}(B[1])$-comodule $M[1]\otimes \mathrm{T} (B[1])$
cogenerated by $M$. The Taylor components are $\mathrm{d}_M^j\colon
M[1]\otimes B[1]^{\otimes j}\to M[1]$ and in the unital case we
require that $\mathrm{d}_M^1(m,s1)=(-1)^{|m|-1}m$ and
$\mathrm{d}_M^j(m,b_1,\dots,b_j)=0$ if some $b_j$ is $s1$. Left
modules are defined similarly. An $A_\infty$-$A$-$B$-bimodule $M$
over $A_\infty$-algebras $A$, $B$ is the datum of a codifferential
on the $\mathrm{T}(A[1])$-$\mathrm{T}(B[1])$-bicomodule
$\mathrm{T}(A[1])\otimes M[1]\otimes \mathrm{T}(B[1])$, given by its
Taylor components $\mathrm{d}_M^{j,k}\colon A[1]^{\otimes j}\otimes
M[1]\otimes B[1]^{k}\to M[1]$. The following is a simple but
important observation.

\begin{Lem} If $M$ is an $A_\infty$-$A$-$B$-bimodule and $A$ is a flat
$A_\infty$-algebra then $M$ with Taylor components
$\mathrm{d}_M^{0,k}$ is a right $A_\infty$-module over $B$.
\end{Lem}

Morphisms of $A_\infty$-algebras ($A_\infty$-(bi)modules) are
(degree 0) morphisms of graded counital coalgebras (respectively,
(bi)comodules) commuting with the codifferentials. Morphisms of
tensor coalgebras and of free comodules are again uniquely
determined by their Taylor components. For instance a morphism of
right $A_\infty$-modules $M\to N$ over $B$ is uniquely determined by
the components $f_j\colon M[1]\otimes B[1]^{\otimes j}\to N[1]$.

\begin{Def}
A morphism between cofree (left-, right-, bi-) comodules over the cofree tensor coalgebra is said to be of
{\em finite type} if all but finitely many of its Taylor components
vanish.
Therefore, by abuse of terminology, we may speak of a morphism of finite type between (left-, right-, bi-) $A_\infty$-modules over an $A_\infty$-algebra.
\end{Def}

The identity morphism is of finite type and the composition of
morphisms of finite type is again of finite type.

The unital algebra of endomorphisms of finite type of a right
$A_\infty$-module $M$ over an $A_\infty$-algebra $B$ is the $0$-th cohomology of a differential graded algebra
$\underline{\mathrm{End}}_{-B}(M)=\oplus_{j\in\mathbb
Z}\underline{\mathrm{End}}^j_{-B}(M)$. The component of degree $j$
is the space of endomorphisms of degree $j$ of finite type of the
comodule $M[1]\otimes \mathrm{T}(B[1])$. The differential is the
graded commutator $\delta f=[\mathrm{d}_M,f]=\mathrm{d}_M\circ
f-(-1)^jf\circ \mathrm{d}_M$ for $f\in
\underline{\mathrm{End}}^j_{-B}(M)$. If $M$ is an
$A_\infty$-$A$-$B$-bimodule and $A$ is flat, then
$\underline{\mathrm{End}}_{-B}(M)$ is defined and the left
$A$-module structure induces a {\em left action} $\mathrm{L}_A$,
which is a morphism of $A_\infty$-algebras $A\to
\underline{\mathrm{End}}_{-B}(M)$: its Taylor components are
$\mathrm{L}^j_A(a)^k(m\otimes b)=\mathrm{d}_M^{j,k}(a\otimes
m\otimes b)$, $a\in A[1]^{\otimes j}$, $m\in M[1]$, $b\in
B[1]^{\otimes k}$.

\begin{Lem} Let $M$ be a right $A_\infty$-module over a unital
$A_\infty$-algebra $B$. Then the subspace
$\underline{\mathrm{End}}_{-B^+}(M)$ of endomorphisms $f$ such that
$f_j(m,b_1,\dots,b_j)=0$ whenever $b_i=s1$ for some $i$, is a
differential graded subalgebra.
\end{Lem}

We call this differential graded subalgebra the subalgebra of {\em
normalized} endomorphisms.

\begin{proof}
It is clear from the formula for Taylor components of the
composition that normalized endomorphisms form a graded subalgebra:
$(f\circ g)^k=\sum_{i+j=k} f^j\circ (g^i\otimes
\mathrm{id}_{B[1]}^{\otimes j})$. The formula for the Taylor
components of the differential of an endomorphism $f$ is
\begin{eqnarray*}
 (\delta f)^k&=&
 \sum_{i+j=k}
 \bigl(
 \mathrm{d}_M^j\circ
 (f^i\otimes\mathrm{id}_{B[1]}^{\otimes j})
 -(-1)^{|f|}
 f^i\circ
 (\mathrm{d}_M^j\otimes\mathrm{id}_{B[1]}^{\otimes i})
 \\
 &&-(-1)^{|f|} f^{k-j+1}\circ
 (\mathrm{id}_{M[1]}\otimes
 \mathrm{id}_{B[1]}^{\otimes i}
 \otimes \mathrm{d}_B^j\otimes
 \mathrm{id}_{B[1]}^{\otimes(k-i-j)})\bigr).
\end{eqnarray*}
If $f$ is normalized and $b_i=s1$ for some $i$, then only two terms
contribute non-trivially to $(\delta f)^k(m,b_1,\dots,b_k)$, namely
$f^{k-1}(m,b_1,\dots,\mathrm{d}_B^2(s1,b_{i+1}),\dots)$ (or
$\mathrm{d}_M^1(f^{k-1}(m,b_1,\dots,b_{k-1}),s1)$  if $i=k$) and
$f^{k-1}(m,b_1,\dots, \mathrm{d}_B^2(b_{i-1},s1),\dots)$ (or
$f^{k-1}(\mathrm{d}_M^1(m,s1),b_2,\dots)$ if $i=1$). Due to the unital
condition these two terms are equal up to sign, hence cancel together.
\end{proof}

The same definitions apply to $A_\infty$-algebras and $A_\infty$-bimodules 
over $\mathbb C[\![\hbar]\!]$ with completed tensor products
and continuous homomorphisms for the $\hbar$-adic topology, so that
for vector spaces $V,W$ we have
$V[\![\hbar]\!]\otimes_{\mathbb{C}[\![\hbar]\!]}W[\![\hbar]\!]=(V\otimes_{\mathbb
C} W)[\![\hbar]\!]$ and $\mathrm{Hom}_{\mathbb
C[\![\hbar]\!]}(V[\![\hbar]\!],W[\![\hbar]\!])=\mathrm{Hom}_{\mathbb
C}(V,W)[\![\hbar]\!]$. A flat deformation of an $A_\infty$-algebra $B$
is an $A_\infty$-algebra $B_\hbar$ over $\mathbb C[\![\hbar]\!]$ which,
as a $\mathbb C[\![\hbar]\!]$-module, is isomorphic to $B[\![\hbar]\!]$ and
such that $B_\hbar/\hbar B_\hbar\simeq B$. Similarly we have flat
deformations of (bi)modules. A right $A_\infty$-module $M_\hbar$ over
$B_\hbar$ which is a flat deformation of $M$ over $B$ is given by
Taylor coefficients $\mathrm{d}_{M_\hbar}^j\in\mathrm{Hom}_{\mathbb
C}(M[1]\otimes B[1]^{\otimes j},M[1])[\![\hbar]\!]$. The differential
graded algebra $\underline{\mathrm{End}}_{-B_\hbar}(M_\hbar)$ of
endomorphism of finite type is then defined as the direct sum of the
homogeneous components of $\mathrm{End}^{\mathrm{
finite}}_{\mathrm{comod}-\mathrm{T}(B[1])}(M[1]\otimes
\mathrm{T}(B[1]))[\![\hbar]\!]$ with differential
$\delta_\hbar=[\mathrm{d}_{M_\hbar},\ ]$. Thus its degree $j$ part
is the $\mathbb C[\![\hbar]\!]$-module
\[
\underline{\mathrm{End}}^j_{B_\hbar}(M_\hbar)=
\left(\oplus_{k\geq0}\mathrm{Hom}^j(M[1]\otimes B[1]^{\otimes
k},M[1])\right)\![\![\hbar]\!],
\]
where $\mathrm{Hom}^j$ is the space of homomorphisms of degree $j$
between graded vector spaces over $\mathbb C$.

Finally, the following notation will be used: if $\phi\colon
V_1[1]\otimes\cdots V_n[1]\to W[1]$ is a linear map and $V_i,W$ are
graded vector spaces or free $\mathbb C[\![\hbar]\!]$-modules, we set
\[
\phi(v_1|\cdots|v_n)=s^{-1}\phi(sv_1\otimes\cdots\otimes
sv_n),\qquad v_i\in V_i.
\]

\subsection{Formality theorem for two branes and deformation of
bimodules}\label{ss-2-2} 
Let $A=\mathrm{S}(V)$ be the symmetric algebra of a
finite dimensional vector space $V$, viewed as a graded algebra
concentrated in degree 0. Let $B=\wedge (V^*)=\mathrm{S}(V^*[-1])$
be the exterior algebra of the dual space with $\wedge^i(V^*)$ of
degree $i$~\footnote{In the case at hand, $V$ is a graded vector space concentrated in degree $0$ and the identification $\wedge(V^*)=\mathrm S(V^*[-1])$ as {\bf graded algebras} is canonical. For a more general graded vector space $V$, $\mathrm S(V^*[-1])$ and $\wedge(V^*)$ are different objects; still, $\mathrm S^n(V^*[-1])$ is canonically isomorphic to $\wedge^n(V^*)[-n]$ for every $n$ by the {\em d\'ecalage} isomorphism, which is simply the identity when $V$ is concentrated in degree $0$.}.
For any graded vector space $W$, the augmentation module
over $\mathrm S(W)$ is the unique one-dimensional module on which
$W$ acts by $0$. Let $A_\hbar=(A[\![\hbar]\!],\star)$ be the Kontsevich
deformation quantization of $A$ associated with a polynomial Poisson bivector
field $\hbar\pi$. It is an associative algebra over $\mathbb
C[\![\hbar]\!]$ with unit $1$. The graded version of the formality
theorem, applied to the same Poisson bracket (more precisely, to the image of $\hbar\pi$ w.r.t.\ the isomorphism of dg Lie algebras $T_\mathrm{poly}(A)\cong T_\mathrm{poly}(B)$), also defines a
deformation quantization $B_\hbar$ of the graded commutative algebra
$B$. However $B_\hbar$ is in general a unital $A_\infty$-algebra
with non-trivial Taylor components $\mathrm{d}_{B_\hbar}^k$ for all
$k$ including $k=0$. Still, the differential graded algebra
$\underline{\mathrm{End}}_{-B_\hbar}(M_\hbar)$ is defined since
$A_\hbar$ is an associative algebra and thus a flat
$A_\infty$-algebra. The following result is a consequence of the
formality theorem for two branes (=submanifolds) in an affine space,
in the special case where one brane is the whole space and the other
a point, and is proved in \cite{CFFR}. It is a version of the Koszul
duality between $A_\hbar$ and $B_\hbar$.

\begin{Prop}\label{p-kosz}
Let $A=\mathrm{S}(V)$, $B=\wedge (V^*)$ for some finite dimensional
vector space $V$ and let $A_\hbar$, $B_\hbar$ be their deformation
quantizations corresponding to a polynomial Poisson bracket.
\begin{enumerate}
\item[$(i)$]
There exists a one-dimensional $A_\infty$-$A$-$B$-bimodule $K$, which, as a left $A$-module and as a right $B$-module, is the augmentation
module, and such that $\mathrm{L}_A\colon A\to
\underline{\mathrm{End}}_{-B}(K)$ is an $A_\infty$-quasi-isomorphism.
\item[$(ii)$]
The bimodule $K$ admits a flat deformation $K_\hbar$ as an
$A_\infty$-$A_\hbar$-$B_\hbar$-bimodule such that
$\mathrm{L}_{A_\hbar}\colon A_\hbar\to
\underline{\mathrm{End}}_{-B_\hbar}(K_\hbar)$ is an
$A_\infty$-quasi-isomorphism.
\item[$(iii)$]
The $A_\infty$-$A_\hbar$-$B_\hbar$-bimodule $K_\hbar$ is in particular a right $A_\infty$-module over the
unital $A_\infty$-algebra $B_\hbar$. The first Taylor component
$\mathrm{L}_{A_\hbar}^1$ sends $A_\hbar$ to the differential graded
subalgebra $\underline{\mathrm{End}}_{-B^+_\hbar}(K_\hbar)$ of
normalized endomorphisms.
\end{enumerate}
\end{Prop}
The proof of (i) and (ii) is contained in \cite{CFFR}. The claim
(iii) follows from the explicit form of the Taylor components
$\mathrm{d}_{K_\hbar}^{1,j}$, given in \cite{CFFR}, appearing in the
definition of $\mathrm{L}^1_A$:
\[
\mathrm{L}^1_{A_\hbar}(a)^j(1|b_1|\cdots|b_j)=\mathrm{d}_{K_\hbar}^{1,j}(a|1|b_1|\dots|b_j).
\]
Namely $\mathrm{d}^{1,j}_{K_\hbar}$ is a power series in $\hbar$
whose term of degree $m$ is a sum over certain directed graphs with $m$ vertices in the complex upper half-plane (vertices of the first type) and $j+2$ ordered vertices on the real axis (vertices of the second type).
To each vertex of the first type is associated a copy of $\hbar\pi$; to the first vertex of the second type is associated $a$, to the second $1$, and to the remaining $j$ vertices the elements $b_i$.
An example of such a graph is depicted in Figure 4, Subsection~\ref{ss-3-2}.

Each graph contributes a multidifferential
operator acting on $a,b_1,\dots,b_j$ times a weight, which is an
integral of a differential form on a compactified configuration space of $m$
points in the complex upper half-plane and $j+2$ ordered points on the real axis modulo dilations and real translations. 
The convention is that to each directed edge of such a graph is associated a derivative acting on the element associated to the final point of the said edge and a $1$-form on the corresponding compactified configuration space.

Therefore, since each $b_i$ may be regarded as a constant polyvector field on $V^*$, there is no edge with final point at a vertex of the second type where a $b_i$ sits (and obviously also where the constant function $1$ sits).
If $j\geq 1$ and $b_i$ belongs to $\mathbb C$ for some $1\leq i\leq j$, the vertex of the second type where $b_i$ sits is neither the starting nor the final point of any directed edge: since $j\geq 1$, the dimension of the corresponding compactified configuration space is strictly positive.
We may use dilations and real translations to fix vertices (of the first and/or second type) distinct from the one where $b_i$ sits: thus, there would be a $1$-dimensional submanifold (corresponding to the interval, where the vertex corresponding to $b_i$ sits), over which there is nothing to integrate, hence the corresponding weight vanishes.

We turn to the description of the differential graded algebra
$\underline{\mathrm{End}}^j_{-B^+_\hbar}(K_\hbar)$. Let
$B^+=\oplus_{j\geq1}\wedge^j(V^*)=\wedge(V^*)/\mathbb C$.
We have
\[
\underline{\mathrm{End}}^j_{-B^+_\hbar}(K_\hbar)= (\oplus_{k\geq
0}\mathrm{Hom}^j(K[1]\otimes B^+[1]^{\otimes k},K[1]))[\![\hbar]\!],
\]
with product
\[
(\phi\cdot\psi)\!(1|b_1|\cdots|b_n)=\sum_k
\psi(1|b_1|\dots|b_k)\phi(1|b_{k+1}|\cdots|b_n).
\]
It follows that the algebra
$\underline{\mathrm{End}}^j_{-B^+_\hbar}(K_\hbar)$ is isomorphic to
the tensor algebra $\mathrm{T} (B^+[1]^*)[\![\hbar]\!]$ generated by
$\mathrm{Hom}(K[1]\otimes B^+[1],K[1])\simeq B^+[1]^*$. In
particular it is concentrated in non-positive degrees.

\begin{Lem}
The restriction $\delta_\hbar\colon B^+[1]^*\to
\mathrm{T}(B^+[1]^*)[\![\hbar]\!]$ of the differential of
$\underline{\mathrm{End}}_{-B^+_\hbar}(K_\hbar) \simeq \mathrm{T}
(B^+[1]^*)[\![\hbar]\!]$ to the generators is dual to the $A_\infty$-structure $\mathrm d_{B_\hbar}$ in the sense that
\[
(\delta_\hbar f)^k(z\otimes b)=(-1)^{|f|}f(z\otimes
\mathrm{d}^k_{B_\hbar}(b)),\ z\in K[1], \ b\in
B[1]^{\otimes k},
\]
for any $f\in\mathrm{Hom}(K[1]\otimes B^+[1],K[1])\simeq B^+[1]^*$
\end{Lem}

\begin{proof} The $A_\infty$-structure of $B_\hbar$ is given by
Taylor components $\mathrm{d}^k_{B_\hbar}\colon B[1]^{\otimes k}\to
B[1]$. By definition the differential on
$\underline{\mathrm{End}}^j_{-B^+_\hbar}(K_\hbar)$ is the graded
commutator $\delta_\hbar f=[\mathrm{d}_{K_\hbar},f]$. In terms of
Taylor components,
\begin{align*}
(\delta_\hbar
f)^k(z\otimes b_1\otimes\cdots\otimes b_k)&=\mathrm{d}^{k-1}_{K_\hbar}(f(z\otimes b_1)\otimes b_{2}\otimes \cdots\otimes b_k)-(-1)^{|f|}f(\mathrm{d}^{k-1}_{K_\hbar}(z\otimes b_1\otimes \cdots\otimes b_{k-1})\otimes b_k)+\\
&\phantom{=}+(-1)^{|f|}f(z\otimes \mathrm{d}^k_{B_\hbar}(b_{1}\otimes
\cdots\otimes b_{k})).
\end{align*}
The first two terms vanish if $b_i\in B^+[1]$ for degree reasons.
\end{proof}

Thus $\mathrm{L}_{A_\hbar}$ induces an isomorphism from $A_\hbar$ to
the cohomology in degree $0$ of
$\underline{\mathrm{End}}_{-B^+_\hbar}(K_\hbar)\simeq \mathrm{T}
(B^+[1]^*)[\![\hbar]\!]$.

\begin{Rem} For $\hbar=0$ this complex is Adam's cobar construction
of the graded coalgebra $B^*$, which is a free resolution of $\mathrm
{S}(V)$.
\end{Rem}

\begin{Thm}\label{t-sh}
The composition
\[\mathrm{L}^1_{A_\hbar}\colon A_\hbar\to
\underline{\mathrm{End}}_{-B^+_\hbar}(K_\hbar)\stackrel\simeq\to
\mathrm{T}(B^+[1]^*)[\![\hbar]\!],
\]
induces on cohomology an algebra isomorphism
\[
\mathrm{L}^1_{A_\hbar}\colon A_\hbar\to \mathrm{T}(V)/\left(\mathrm{T}(
V)\otimes \delta_{\hbar}((\wedge^2V^*)^*) \otimes \mathrm{T}(V)\right),
\]
where $\delta_{\hbar}\colon (\wedge^2V^*)^*\to
\mathrm{T}(V)[\![\hbar]\!]$ is dual to $\oplus_{k\geq 0}
\mathrm{d}^k_{B_\hbar}\colon (B^+[1]^0)^{\otimes k}=V^{\otimes k}\to
B^+[1]^1=\wedge^2 V^*$.
\end{Thm}

\begin{proof}
The fact that the map is an isomorphism follows from the fact that
it is so for $\hbar=0$, by the classical Koszul duality. As the
cohomology is concentrated in degree $0$ it remains so for the
deformed differential $\delta_\hbar$ over $\mathbb C[\![\hbar]\!]$.

As a graded vector space, $B^+[1]^*=V\oplus
(\wedge^2V^*)^*\oplus\cdots$, with $(\wedge^iV^*)^*$ in degree
$1-i$. Therefore the complex $\mathrm{T}(B^+[1]^*)[\![\hbar]\!]$ is
concentrated in non-positive degrees and begins with
\[
\cdots\to \left(\mathrm{T}(V)\otimes (\wedge^2 V^*)^*\otimes
\mathrm{T}(V)\right)[\![\hbar]\!]\to \mathrm{T}(V)[\![\hbar]\!]\to 0.
\]
Thus to compute the degree $0$ cohomology we only need the
restriction of the Taylor components $\mathrm{d}^k_{B_\hbar}$ on
$\mathrm{T}(V^*)=\mathrm{T} (B^+[1])^0$, whose image is in
$B[1]^1=\wedge^2V^*$.
\end{proof}

This theorem gives a presentation of the algebra $A_\hbar$ by
generators and relations. Let $x_1,\dots,x_d\in V$ be a system of
linear coordinates on $V^*$ dual to a basis $e_1,\dots,e_d$. Let for
$I=\{i_1<\cdots<i_k\}\subset\{1,\dots,d\}$, $x_I\in(\wedge^kV^*)^*$
be dual to the basis $e_{i_1}\wedge\cdots\wedge e_{i_k}$. Then
$A_\hbar$ is isomorphic to the algebra generated by $x_1,\dots,x_d$
subject to the relations $\delta_\hbar(x_{ij})=0$. Up to order
$1$ in $\hbar$ the relations are obtained from the cobar
differential and the graph of Figure 1.
\[
\delta_\hbar(x_{ij})=x_i\otimes x_j-x_j\otimes
x_i-\hbar\mathrm{Sym}(\pi_{ij})+O(\hbar^2).
\]
Here $\mathrm{Sym}$ is the symmetrization map $\mathrm{S}(V)\to
\mathrm{T}(V)$.
\bigskip
\begin{center}
\resizebox{0.4 \textwidth}{!}{\input{dual_deg2-1.pstex_t}}\\
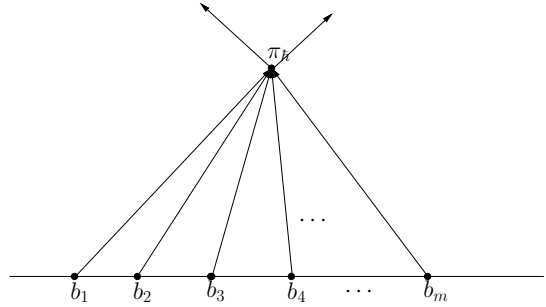
\text{Figure 1 -  The only admissible graph contributing to $\mathrm{d}_{B_\hbar}^m$ at order $1$ in $\hbar$ } \\
\end{center}
\bigskip
The lowest order of the isomorphism induced by $\mathrm{L}^1_A$ on
generators $x_i\in V$ of $A_\hbar=\mathrm{S}(V)[\![\hbar]\!]$ was
computed in \cite{CFFR}:
\[
\mathrm{L}^1_A(x_i)=x_i+O(\hbar).
\]
The higher order terms $O(\hbar)$ are in general non-trivial (for
example in the case of the dual of a Lie algebra, see below).

By comparing our construction with the arguments in~\cite{Sh2}, we
see that the differential $\mathrm{d}_\hbar$ corresponds to the
image of $\mathcal V(\widehat{\pi_\hbar})$, where the notations are
as in the introduction, by the quasi-isomorphism $\Phi_1$
in~\cite[Subsection 1.4]{Sh2}. Hence, Theorem~\ref{t-sh} provides a
proof of~\cite[Conjecture 2.6]{Sh2} with the amendment that the
isomorphism $A_\hbar\to \mathrm{T}(V)/ \mathcal I_\star$ is not just
given by the symmetrization map but has non-trivial corrections.

\section{Examples}\label{s-3}
We now want to examine more closely certain special cases of
interest. We assume here that the reader has some familiarity with
the graphical techniques of \cites{K, CF, CFFR}. To obtain the
relations $\delta_\hbar(x_{ij})$ we need $\mathrm
d^m_{B_\hbar}(b_1|\cdots|b_m)\in\wedge^2V^*[\![\hbar]\!]$, for $b_i\in
V^*\subset B^+$. The contribution at order $n$ in $\hbar$ to this is
given by a sum over the set $\mathcal G_{n,m}$ of admissible graphs
with $n$ vertices of the first type and $m$ of the second type.

\subsection{The Moyal--Weyl product on $V$}\label{ss-3-1}

Let $\pi_\hbar=\hbar \pi$ be a constant Poisson bivector on $V^*$,
which is uniquely characterized by a complex, skew-symmetric matrix
$d\times d$-matrix $\pi_{ij}$.

In this case, Kontsevich's deformed algebra $A_\hbar$ has an
explicit description: the associative product on $A_\hbar$ is the
Moyal--Weyl product
\[
(f_1\star f_2)=\mathrm m \circ \exp{\frac{1}2 \pi_\hbar},
\]
where $\pi_\hbar$ is viewed here as a bidifferential operator, the
exponential has to be understood as a power series of bidifferential
operators, and $\mathrm m$ denotes the ($\mathbb C[\![\hbar]\!]$-linear)
product on polynomial functions on $V^*$. On the other hand, it is
possible to compute explicitly the complete $A_\infty$-structure on
$B_\hbar$.
\begin{Lem}\label{l-weyl}
For a constant Poisson bivector $\pi_\hbar$ on $V^*$, the $A_\infty$-structure on $B_\hbar$ has only two non-trivial Taylor components, namely
\begin{equation}\label{eq-weyl}
\mathrm d_{B_\hbar}^0(1)=\hbar \pi,\quad \mathrm d_{B_\hbar}^2(b_1|b_2)=(-1)^{|b_1|}b_1\wedge b_2,\quad b_i\in B_\hbar,\quad i=1,2.
\end{equation}
\end{Lem}
\begin{proof}
We consider $\mathrm d^m_{B_\hbar}$ first in the case $m=0$. Admissible
graphs contributing to $\mathrm d_{B_\hbar}^0$ belong to $\mathcal
G_{n,0}$, for $n\geq 1$. For $n\geq 2$, all graphs give
contributions involving a derivative of $\pi_{ij}$ and thus vanish.
There remains the only graph in $\mathcal G_{1,0}$, whence the first
identity in~\eqref{eq-weyl}.

By the same reasons, $\mathrm d_{B_\hbar}^m$ is trivial, if $m\geq
1$ and $m\neq 2$: in the case $m=1$, we have to consider
contributions coming from admissible graphs in $\mathcal G_{n,1}$,
with $n\geq 1$, which vanish for the same reasons as in the case
$m=0$.

For $m\geq 3$, contributions coming from admissible graphs in
$\mathcal G_{n,m}$, $n\geq 1$, are trivial by a dimensional
argument.

Finally, once again, the only possibly non-trivial contribution
comes from the unique admissible graph in $\mathcal G_{0,2}$ which
gives the product.
\end{proof}
As a consequence, the differential $\delta_\hbar$ can be explicitly
computed, namely
\[
\delta_{\hbar}(x_{ij})=x_i\otimes x_j-x_j\otimes
x_i-\hbar\pi_{ij}.
\]
This provides the description of the Moyal--Weyl algebra as the
algebra generated by $x_i$ with relations $[x_i,x_j]=\hbar\pi_{ij}$.

We finally observe that the quasi-isomorphism $\mathrm L_{A_\hbar}^1$ coincides, by a direct computation, with the usual symmetrization morphism.

\subsection{The universal enveloping algebra of a finite-dimensional Lie algebra $\mathfrak g$}\label{ss-3-2}
We now consider a finite-dimensional complex Lie algebra
$V=\mathfrak g$: its dual space $\mathfrak g^*$ with
Kirillov--Kostant-Souriau Poisson structure. With respect to a basis
$\{x_i\}$ of $\mathfrak g$, we have
\[
\pi=f_{ij}^k x_k\partial_i\wedge \partial_j,
\]
where $f_{ij}^k$ denote the structure constant of $\mathfrak g$ for
the chosen basis.

It has been proved in~\cite[Subsubsection 8.3.1]{K} that
Kontsevich's deformed algebra $A_\hbar$ is isomorphic to the
universal enveloping algebra $\mathrm U_\hbar(\mathfrak g)$ of
$\mathfrak g[\![\hbar]\!]$ for the $\hbar$-shifted Lie bracket $\hbar[\
,\ ]$.

On the other hand, we may, once again, compute explicitly the $A_\infty$-structure on $B_\hbar$.
\begin{Lem}\label{l-CE}
The $A_\infty$-algebra $B_\hbar$ determined by $\pi_\hbar$, where
$\pi$ is the Kirillov--Kostant--Souriau Poisson structure on
$\mathfrak g^*$, has only two non-trivial Taylor components, namely
\begin{equation}\label{eq-CE}
\mathrm d_{B_\hbar}^1(b_1)=\mathrm d_{\mathrm{CE}}(b_1),\quad \mathrm d_{B_\hbar}^2(b_1|b_2)=(-1)^{|b_1|}b_1\wedge b_2,\quad b_i\in B_\hbar,\quad i=1,2,
\end{equation}
where $\mathrm d_{\mathrm{CE}}$ denotes the Chevalley--Eilenberg differential of $\mathfrak g$, endowed with the rescaled Poisson bracket $\hbar[\bullet,\bullet]$.
\end{Lem}
\begin{proof}
By dimensional arguments and because of the linearity of
$\pi_\hbar$, there are only two admissible graphs in $\mathcal
G_{1,0}$ and $\mathcal G_{2,0}$, which may contribute non-trivially
to the curvature of $B_\hbar$, namely,
\bigskip
\begin{center}
\resizebox{0.6 \textwidth}{!}{\input{CE-graph.pstex_t}}\\
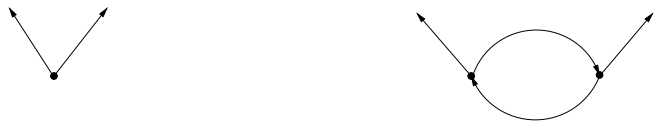
\text{Figure 2 -  The only admissible graphs in $\mathcal G_{1,0}$ and $\mathcal G_{2,0}$ respectively in the curvature of $B_\hbar$ } \\
\end{center}
\bigskip
The operator $\mathcal O_\Gamma^B$ for the graph in $\mathcal
G_{1,0}$ vanishes, when setting $x=0$. On the other hand, $\mathcal
O_\Gamma^B$ vanishes in virtue of~\cite[Lemma 7.3.1.1]{K}.

We now consider the case $m\geq 1$. We consider an admissible graph
$\Gamma$ in $\mathcal G_{n,m}$ and the corresponding operator
$\mathcal O_\Gamma^B$: the degree of the operator-valued form
$\omega_\Gamma^B$ equals the number of derivations acting on the
different entries associated to vertices either of the first or
second type. Thus, the operator $\mathcal O_\Gamma^B$ has a
polynomial part (since all structures are involved are polynomial on
$\mathfrak g^*$): since the polynomial part of any of its arguments
in $B_\hbar$ has degree $0$, the polynomial degree of $\mathcal
O_\Gamma^B$ must be also $0$. A direct computation shows that this
condition is satisfied if and only if $n+m=2$, because $\pi_\hbar$
is linear.

Obviously, the previous identity is never satisfied if $m\geq 3$,
which implies immediately that the only non-trivial Taylor
components appear when $m=1$ and $m=2$. When $m=1$, the previous
equality forces $n=1$: there is only one admissible graph $\Gamma$
in $\mathcal G_{1,1}$, whose corresponding operator is non-trivial,
namely,
\bigskip
\begin{center}
\resizebox{0.35 \textwidth}{!}{\input{CE-graph-1.pstex_t}}\\
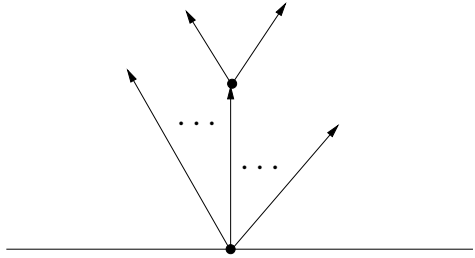
\text{Figure 3 -  The only admissible graph in $\mathcal G_{1,1}$ contributing to $\mathrm d_{B_\hbar}^1$ } \\
\end{center}
\bigskip
The weight is readily computed, and the identification with the
Chevalley--Eilenberg differential is then obvious.

Finally, when $m=2$, the result is clear by previous computations.
\end{proof}
Thus $\delta_\hbar$ is given by
\[
\delta_\hbar(x_{ij})= x_{i}\otimes x_{j}-x_{j}\otimes x_{i}-\hbar
\sum_{k}f_{ij}^k x_k.
\]
Hence we reproduce the result that $A_\hbar$ is isomorphic to
$\mathrm U_\hbar(\mathfrak g)$.
We now want to give an explicit expression for the isomorphism $\mathrm L^1_{A_\hbar}$.

We consider the expression $\mathrm L^1_{A_\hbar}(a)^m(1|b_1|\cdots|b_m)=\mathrm d_{K_\hbar}^{1,m}(a|1|b_1|\cdots|b_m)$.
Degree reasons imply that the sum of the degrees of the elements $b_i$ equals $m$; furthermore, if the degree of some $b_i$ is strictly bigger than $1$, the previous equality forces a different $b_j$ to have degree $0$, whence the corresponding expression vanishes by Proposition~\ref{p-kosz}, $(iii)$.
Hence, the degree of each $b_i$ is precisely $1$.
We now consider a general graph $\Gamma$ with $n$ vertices of the first type and $m+2$ ordered vertices of the second type; to each vertex of the first type is associated a copy of $\hbar\pi$, while to the ordered vertices of the second type are associated $a$, $1$ and the $b_i$s in lexicographical order.
We denote by $p$ the number of edges departing from the $n$ vertices of the first type and hitting the first vertex of the second type (observe that in this situation edges departing from vertices of the first type can only hit vertices of the first type or the first vertex of the second type): in the present framework, edges have only one color (we refer to~\cite[Section 7]{CFFR} and~\cite[Subsection 3.2]{CRT} for more details on the $4$-colored propagators and corresponding superpropagators entering the $2$ brane Formality Theorem), thus there can be {\bf at most} one edge hitting the first vertex of the second type, whence $p\leq n$.
We now compute the polynomial degree of the multidifferential operator associated to the graph $\Gamma$: it equals $n-j-(2n-p)=p-j-n$, where $0\leq j\leq m$ is the number of edges from the last $m$ vertices of the second type hitting vertices of the first type.
The first $n$ comes from the fact that $\pi$ is a linear bivector field.
As $p-j-n\geq 0$ and $p\leq n$, it follows immediately $p=n$ and $j=0$, {\em i.e.} the edges departing from the last $m$ vertices of the second type all hit the first vertex of the second type, and from each vertex of the first type departs exactly one edge hitting the first vertex of the second type; the remaining $n$ edges must hit a vertex of the first type.

In summary, a general graph $\Gamma$ appearing in $\mathrm L_{A_\hbar}^1(a)(1|b_1|\cdots|b_m)$ is the disjoint union of wheel-like graphs $\mathcal W_n$, $n\geq 1$, and of the graph $\beta_m$, $m\geq 0$; such graphs are depicted in Figure 4.

Observe that the $1$-wheel $\mathcal W_1$ appears here explicitly because of the presence of short loops in the $2$ brane Formality Theorem~\cite[]{CFFR}: the integral weight of the $1$-wheel has been computed in~\cite{CRT} and equals $-1/4$, while the corresponding translation invariant differential operator is the trace of the adjoint representation of $\mathfrak g$.
Any multiple of $c_1=\mathrm{tr}_\mathfrak g\circ\mathrm{ad}$ defines a constant vector field on $\mathfrak g^*$: either as an easy consequence of the Formality Theorem of Kontsevich~\cite{K} or by an explicit computation using Stokes' Theorem, $c_1$ is a derivation of $(A_\hbar,\star)$, where $\star$ is the deformed product on $A_\hbar$ {\em via} Kontsevich's deformation quantization.

The integral weight of the graph $\beta_m$ is $1/m!$ and the corresponding multidifferential operator is simply the symmetrization morphism; the integral weight of the wheel-like graph $\mathcal W_n$, $n\geq 2$, has been computed in~\cite{W, VdB} (observe that, except the case $n=1$, the integral weights of $\mathcal W_n$ for $n$ odd vanish) and equal the modified Bernoulli numbers, and the corresponding translation-invariant differential operators are $c_n=\mathrm{tr}_\mathfrak g(\mathrm{ad}^n(\bullet))$.
\bigskip
\begin{center}
\resizebox{0.45 \textwidth}{!}{\input{Wheel_symm.pstex_t}}\\
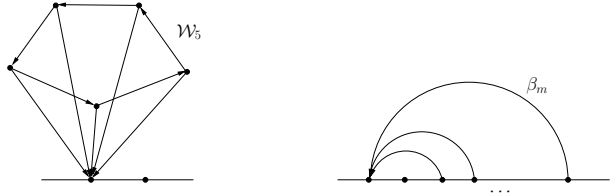
\text{Figure 4 - The wheel $\mathcal W_5$ with $5$ spokes and the graph $\beta_m$} \\
\end{center}
\bigskip 
Therefore, the isomorphism $\mathrm L_{A_\hbar}^1$ (for $\hbar=1$) equals the composition of the PBW isomorphism from $\mathrm S(\mathfrak g)$ to $\mathrm U(\mathfrak g)$ with Duflo's strange automorphism; the derivation $-1/4\ c_1$ of the deformed algebra $(A,\star)$ is exponentiated to an automorphism of the same algebra.
(The fact that $\pi$ is linear permits to set $\hbar=1$, see also~\cite[Subsubsection 8.3.1]{K} for an explanation.)



\subsection{Quadratic algebras}\label{ss-3-3}
Here we briefly discuss the case where $V^*$ is endowed with a
quadratic Poisson bivector field $\pi$: this case has been already
considered in detail in~\cite[Section 8]{CFFR}, see also~\cite{Sh},
where the property of the deformation associated $\pi_\hbar$ of
preserving the property of being Koszul has been proved.

The main feature of the quadratic case is the degree $0$ homogeneity
of the Poisson bivector field, which reflects itself in the
homogeneity of all structure maps. In particular the Kontsevich
star-product on a basis of linear functions has the form
\[
x_i\star x_j=x_ix_j +\sum_{k,l} S^{kl}_{ij}(\hbar)x_kx_l,
\]
for some $S^{kl}_{ij}\in\hbar\mathbb C[\![\hbar]\!]$. Our results
implies that this algebra is isomorphic to the quotient of the
tensor algebra in generators $x_i$ by relations
\[
x_i\otimes x_j-x_j\otimes x_i=\sum_{k,l}R^{kl}_{ij}(\hbar)x_k\otimes
x_l,
\]
for some $R^{kl}_{ij}(\hbar)\in\hbar\mathbb C[\![\hbar]\!]$. The
isomorphism sends $x_i$ to
\[
\mathrm L_{A_\hbar}(x_i)=x_i+\sum_j L^j_i(\hbar)x_j,
\]
for some $L^j_i(\hbar)\in\hbar\mathbb C[\![\hbar]\!]$.

\subsection{A final remark}\label{ss-3-4}
We point out that, in~\cite{BG}, the authors construct a flat
$\hbar$-deformation between a so-called non-homogeneous quadratic
algebra and the associated quadratic algebra: the characterization
of the non-homogeneous quadratic algebra at hand is in terms of two
linear maps $\alpha$,  $\beta$, from $R$ onto $V$ and $\mathbb C$
respectively, which satisfy certain cohomological conditions. In the
case at hand, it is not difficult to prove that the conditions on
$\alpha$ and $\beta$ imply that their sum defines an affine Poisson
bivector on $V^*$: hence, instead of considering $\alpha$ and
$\beta$ separately, as in~\cite{BG}, we treat them together. Both
deformations are equivalent, in view of the uniqueness of flat
deformations yielding the PBW property, see~\cite{BG}.

\end{document}

%% file: dual_deg2-1.pstex_t
\begin{picture}(0,0)%
\epsfig{file=dual_deg2-1.pstex}%
\end{picture}%
\setlength{\unitlength}{3947sp}%
\begingroup\makeatletter\ifx\SetFigFont\undefined%
\gdef\SetFigFont#1#2#3#4#5{%
  \reset@font\fontsize{#1}{#2pt}%
  \fontfamily{#3}\fontseries{#4}\fontshape{#5}%
  \selectfont}%
\fi\endgroup%
\begin{picture}(7674,4323)(2239,-6757)
\put(7051,-6661){\makebox(0,0)[lb]{\smash{{\SetFigFont{20}{24.0}{\rmdefault}{\mddefault}{\updefault}{\color[rgb]{0,0,0}$\cdots$}%
}}}}
\put(3121,-6676){\makebox(0,0)[lb]{\smash{{\SetFigFont{20}{24.0}{\rmdefault}{\mddefault}{\updefault}{\color[rgb]{0,0,0}$b_1$}%
}}}}
\put(4021,-6676){\makebox(0,0)[lb]{\smash{{\SetFigFont{20}{24.0}{\rmdefault}{\mddefault}{\updefault}{\color[rgb]{0,0,0}$b_2$}%
}}}}
\put(5071,-6661){\makebox(0,0)[lb]{\smash{{\SetFigFont{20}{24.0}{\rmdefault}{\mddefault}{\updefault}{\color[rgb]{0,0,0}$b_3$}%
}}}}
\put(6226,-6676){\makebox(0,0)[lb]{\smash{{\SetFigFont{20}{24.0}{\rmdefault}{\mddefault}{\updefault}{\color[rgb]{0,0,0}$b_4$}%
}}}}
\put(8176,-6661){\makebox(0,0)[lb]{\smash{{\SetFigFont{20}{24.0}{\rmdefault}{\mddefault}{\updefault}{\color[rgb]{0,0,0}$b_m$}%
}}}}
\put(5941,-3226){\makebox(0,0)[lb]{\smash{{\SetFigFont{20}{24.0}{\rmdefault}{\mddefault}{\updefault}{\color[rgb]{0,0,0}$\pi_\hbar$}%
}}}}
\put(6391,-5641){\makebox(0,0)[lb]{\smash{{\SetFigFont{20}{24.0}{\rmdefault}{\mddefault}{\updefault}{\color[rgb]{0,0,0}$\cdots$}%
}}}}
\end{picture}%

%% file: CE-graph.pstex_t
\begin{picture}(0,0)%
\epsfig{file=CE-graph.pstex}%
\end{picture}%
\setlength{\unitlength}{3947sp}%
\begingroup\makeatletter\ifx\SetFigFont\undefined%
\gdef\SetFigFont#1#2#3#4#5{%
  \reset@font\fontsize{#1}{#2pt}%
  \fontfamily{#3}\fontseries{#4}\fontshape{#5}%
  \selectfont}%
\fi\endgroup%
\begin{picture}(11274,2784)(739,-6373)
\end{picture}%

%% file: CE-graph-1.pstex_t
\begin{picture}(0,0)%
\epsfig{file=CE-graph-1.pstex}%
\end{picture}%
\setlength{\unitlength}{3947sp}%
\begingroup\makeatletter\ifx\SetFigFont\undefined%
\gdef\SetFigFont#1#2#3#4#5{%
  \reset@font\fontsize{#1}{#2pt}%
  \fontfamily{#3}\fontseries{#4}\fontshape{#5}%
  \selectfont}%
\fi\endgroup%
\begin{picture}(4674,2525)(3889,-6414)
\put(6226,-5626){\makebox(0,0)[lb]{\smash{{\SetFigFont{20}{24.0}{\rmdefault}{\mddefault}{\updefault}{\color[rgb]{0,0,0}$\cdots$}%
}}}}
\put(5596,-5191){\makebox(0,0)[lb]{\smash{{\SetFigFont{20}{24.0}{\rmdefault}{\mddefault}{\updefault}{\color[rgb]{0,0,0}$\cdots$}%
}}}}
\end{picture}%

%% file: Wheel_symm.pstex_t
\begin{picture}(0,0)%
\includegraphics{Wheel_symm.pstex}%
\end{picture}%
\setlength{\unitlength}{3947sp}%
\begingroup\makeatletter\ifx\SetFigFont\undefined%
\gdef\SetFigFont#1#2#3#4#5{%
  \reset@font\fontsize{#1}{#2pt}%
  \fontfamily{#3}\fontseries{#4}\fontshape{#5}%
  \selectfont}%
\fi\endgroup%
\begin{picture}(11300,3722)(1163,-6757)
\put(4336,-3676){\makebox(0,0)[lb]{\smash{{\SetFigFont{20}{24.0}{\rmdefault}{\mddefault}{\updefault}{\color[rgb]{0,0,0}$\mathcal W_5$}%
}}}}
\put(10201,-6661){\makebox(0,0)[lb]{\smash{{\SetFigFont{20}{24.0}{\rmdefault}{\mddefault}{\updefault}{\color[rgb]{0,0,0}$\cdots$}%
}}}}
\put(10891,-4681){\makebox(0,0)[lb]{\smash{{\SetFigFont{20}{24.0}{\rmdefault}{\mddefault}{\updefault}{\color[rgb]{0,0,0}$\beta_m$}%
}}}}
\end{picture}%